\documentclass[final,12pt]{colt2020} 
\pdfoutput=1
\usepackage{amsfonts}
\usepackage{calrsfs}
\usepackage{amsmath}
\usepackage{amssymb}
\usepackage{wrapfig}
\usepackage{breakcites}

\DeclareMathAlphabet{\pazocal}{OMS}{zplm}{m}{n}
\newtheorem{prop}{Proposition}[section]
\newtheorem{lem}{Lemma}[section]
\newtheorem{thm}{Theorem}[section]

\newtheorem{defin}{Definition}[section]
\newcommand{\sP}{\mathcal{P}}
\newcommand{\sQ}{\mathcal{Q}}

\newcommand{\sD}{\pazocal{D}}
\newcommand{\sF}{\pazocal{F}}
\newcommand{\sH}{\pazocal{H}}
\newcommand{\sG}{\pazocal{G}}
\newcommand{\spn}{\operatorname{span}}

\newcommand{\bz}{\mathbf{z}}
\newcommand{\bX}{{\mathbf{X}}}

\newcommand{\tGamma}{{\widetilde{\Gamma}}}

\newcommand{\rn}{\mathbb{R}}
\def\simiid{\overset{iid}{\sim}}

\title{Generalized Identifiability Bounds for Mixture Models with Grouped Samples}
\usepackage{times}

\coltauthor{\Name{Robert A. Vandermeulen\textsuperscript{*\dag}} \Email{vandermeulen@tu-berlin.de}\\ \Name{Ren\'e Saitenmacher\textsuperscript{*}} \Email{r.saitenmacher@tu-berlin.de}\\
 \addr  \textsuperscript{\bf*}Machine Learning Group, Technische Universit\"at Berlin\\
 \addr \textsuperscript{\bf \dag}Berlin Institute for the Foundations of Learning and Data\\10587 Berlin, Germany
}

\begin{document}

\maketitle

\begin{abstract}%
  Recent work has shown that finite mixture models with $m$ components are identifiable, while making no assumptions on the mixture components, so long as one has access to groups of samples of size $2m-1$ which are known to come from the same mixture component. In this work we generalize that result and show that, if every subset of $k$ mixture components of a mixture model are linearly independent, then that mixture model is identifiable with only $(2m-1)/(k-1)$ samples per group. We further show that this value cannot be improved. We prove an analogous result for a stronger form of identifiability known as ``determinedness'' along with a corresponding lower bound. This independence assumption almost surely holds if mixture components are chosen randomly from a $k$-dimensional space. We describe some implications of our results for multinomial mixture models and topic modeling.
\end{abstract}


\section{Introduction} \label{sec:intro}
Finite mixture models have seen extensive use in statistics and machine learning. In a finite mixture model one assumes that samples are drawn according to a two-step process. First an unobserved \emph{mixture component}, $\mu$, is randomly selected according to a probability measure over probability measures $\sP= \sum_{i=1}^m a_i \delta_{\mu_i}$ ($\delta$ is the Dirac measure). Next an observed sample $X$ is drawn from $\mu$, $X \sim \mu$. A central question in mixture modeling theory is that of \emph{identifiability} \citep{teicher63}: whether $\sP$ is uniquely determined from the distribution of $X$. From the law of total probability it follows that $X$ is distributed according to $\sum_{i=1}^m a_i \mu_i$. Excepting trivial cases, a mixture model is not identifiable unless one makes additional assumptions about the mixture components. A standard assumption is that the mixture components $\mu_1,\ldots, \mu_m$ are elements of some parametric class of densities. A common choice for this class is the set of multivariate Gaussian distributions, which yields the well-known and frequently-used Gaussian mixture model. This model is indeed known to be identifiable \citep{anderson14,bruni85,yakowitz68}. A natural question to ask is whether it is possible for a mixture model to be identifiable without such parametric assumptions.

In \cite{vandermeulen19} the authors consider an alternative setting for mixture modeling where \emph{no} assumptions are made on the mixture components $\mu_1,\ldots,\mu_m$ and, instead of having access to one sample from each unobserved mixture component $\mu\sim \sP$, one has access to groups of $n$ samples, $\bX = \left(X_1,\ldots,X_n\right)$ which are known to be independently sampled from $\mu$, i.e. $X_1,\ldots,X_n \simiid \mu$. In \cite{vandermeulen19} the authors develop several fundamental bounds relating the identifiability of $\sP$ to the number of samples per group $n$ and the number of mixture components $m$. These bounds consider extremal cases where there are either no assumptions on the mixture components or they are assumed to be linearly independent. If the mixture components are assumed to lie in a finite dimensional space, such as when the sample space is finite, then it is reasonable to assume that the collection of all mixture components is linearly dependent, however sufficiently small subsets of the mixture components are linearly independent.

In this paper we prove two fundamental bounds relating the identifiability of $\sP = \sum_{i=1}^m a_i \delta_{\mu_i}$ to the number of mixture components $m$, the number of samples per group $n$, and a value $k$ which describes the degree of linear independence of the mixture components.  We show that if every subset of $k$ measures in $\mu_1,\ldots,\mu_m$ are linearly independent, then $\sP$ is the simplest mixture, in terms of the number of mixture components, yielding the distribution on $\bX$ if $2m-1\le (k-1)n$. If $n$ is even-valued and $2m-2 \le (k-1)(n-1)$ then $\sP$ is the \emph{only} mixture, with any number of components, yielding the distribution on $\bX$. We furthermore show that the first bound is tight and that the second bound is nearly tight. Most of the bounds in \cite{vandermeulen19} are special cases of the bounds presented in this paper. We also show that this linear independence assumption occurs naturally, similarly to results in \cite{kargas18}, and describe some practical implications of our results.

\section{Background}
We introduce the mathematical setting used in the rest of the paper before reviewing existing results.

\subsection{Problem Setting} \label{ssec:problem-setting}
The setting described here is drawn from \cite{vandermeulen19} and is highly general. Let $\left(\Omega, \sF\right)$ be a $\sigma$-algebra and let $\sD$ be the space of probability measures on $\left(\Omega, \sF\right)$. Note that $\sD$ is contained in the vector space of finite signed measures on $\left(\Omega, \sF\right)$, a fact which we will use often. For an element $\gamma$, let $\delta_\gamma$ denote the Dirac measure at $\gamma$. We equip $\sD$ with the power $\sigma$-algebra. We call a measure $\sP$ on $\sD$ a \emph{mixture of measures} if it is a measure on $\sD$ of the form $\sP = \sum_{i=1}^m a_i \delta_{\mu_i}$ with $a_i>0$, $\sum_{i=1}^m a_i = 1$, and $m < \infty$. We will always assume that the representation of a mixture of measures has minimal $m$, i.e. there are no repeated $\mu_i$ in the summands. For a full technical treatment of the concept of minimal representation see \cite{vandermeulen19}. We refer to the measures $\mu_1,\ldots,\mu_m$ as \emph{mixture components}. We now introduce the model we wish to investigate in this paper which is termed the \emph{grouped sample setting} in \cite{vandermeulen19}. If we let $\mu \sim \sP$ and $X_1,\ldots,X_n\simiid \mu$ then the probability distribution for $\bX = \left(X_1,\ldots,X_n\right)$ is $\sum_{i=1}^m a_i \mu_i^{\times n}$. With this in mind we introduce the following operator,
\begin{equation}\label{eqn:grouped-sample} 
  V_n\left(\sP\right) \triangleq \sum_{i=1}^m a_i \mu_i^{\times n}.
\end{equation}
To give some concreteness to this setting it can be helpful to consider the application of topic modeling with a finite number of topics. Here $\mu_1,\ldots,\mu_m$ are topics, which are simply distributions over words. The measure $\sP$ designates a topic $\mu_i$ being chosen with probability $a_i$. The group of samples $(X_1,\ldots,X_n)$ represent a document containing $n$ words as a bag of words. A collection of documents $\bX_1,\bX_2,\ldots$ are then iid samples of $V_n(\sP)$. In this setting we are interested in the number of words necessary per document to recover the true topic model $\sP$.

We will be investigating two forms of identifiability, \emph{$n$-identifiability} where $\sP$ is the simplest mixture of measures, in terms of the number of mixture components, yielding the distribution on $(X_1,\ldots,X_n)$ and \emph{$n$-determinedness} where $\sP$ is the \emph{only} mixture of measures yielding the distribution on $(X_1,\ldots,X_n)$. We finish this section with the following two definitions which capture these two notions of identifiability.
\begin{defin}\label{def:ident}
  A mixture of measures $\sP = \sum_{i=1}^m a_i \delta_{\mu_i}$ is \emph{$n$-identifiable} if there exists no mixture of measures $\sP' \neq \sP$ with $m$ or fewer components such that $V_n(\sP') = V_n(\sP)$.
\end{defin}
\begin{defin}\label{def:det}
  A mixture of measures $\sP$ is \emph{$n$-determined} if there exists no mixture of measures $\sP'\neq \sP$ such that $V_n\left( \sP' \right) = V_n\left( \sP \right)$.
\end{defin}
\subsection{Previous Results}
Here we recall several results from \cite{vandermeulen19}. In that paper the authors prove five bounds relating identifiability or determinedness to the geometry of the mixture components, the number of mixture components $m$, and the number of samples per group $n$. For brevity we have summarized these bounds in Table \ref{tab:oldresults}. \cite{vandermeulen19} showed that none of these bounds are improvable via matching lower bounds. For clarity we include an example of a precise statement of an entry in this table.

\begin{theorem}[Table \ref{tab:oldresults} Row Four or \cite{vandermeulen19} Theorem 4.6]
  Let $\sP = \sum_{i=1}^m a_i \delta_{\mu_i}$ be a mixture of measures where $\mu_1,\ldots,\mu_m$ are linearly independent. Then $\sP$ is 4-determined.\sloppy
\end{theorem}

\begin{wraptable}{}{9cm}
  \caption{Summary of identifiability results from \cite{vandermeulen19}. $m$ designates the number of mixture components and $n$ the number of samples per group.}\label{tab:oldresults}
  \begin{tabular}{|c|c|c|}
    \hline
    Component assumption& $n$ bound & $n$-ident./det.\\ \hline \hline
    none &$ n\ge 2m-1$& identifiable\\\hline
    none &$ n\ge 2m$& determined\\\hline
    linearly independent &$ n\ge 3$& identifiable\\\hline
    linearly independent &$ n\ge 4$& determined\\\hline
    jointly irreducible &$ n\ge 2$& determined\\\hline
  \end{tabular}
\end{wraptable}
The last row of Table \ref{tab:oldresults} contains a property known as \emph{joint irreducibility} which was introduced in \cite{blanchard14}. A collection of probability measures $\mu_1,\ldots,\mu_m$ is \emph{jointly irreducible} when all probability measures in the linear span of $\mu_1,\ldots,\mu_m$ lie in the convex hull of $\mu_1,\ldots,\mu_m$. We do not use joint irreducibility anywhere else in this paper, however we note that it is a property that is stronger than linear independence. 

For completeness we also include the following lemmas from \cite{vandermeulen19} that demonstrate the unsurprising fact that $k$-identifiability and $k$-determinedness are, in some sense, monotonic. Each lemma encapsulates two statements, one concerning identifiability and one concerning determinedness, which we have combined for brevity.
\begin{lem}\label{lem:ident}
  If a mixture of measures is $n$-identifiable (determined) then it is $q$-identifiable (determined) for all $q>n$.
\end{lem}
\begin{lem} \label{lem:noident}
  If a mixture of measures is not $n$-identifiable (determined) then it is not $q$-identifiable (determined) for any $q<n$.
\end{lem}
Finally \cite{vandermeulen19} Lemma 7.1 showed that if the sample space $\Omega$ is finite, the grouped sample setting is equivalent to a multinomial mixture model where $n$ is the number of trials and $\mu_1,\ldots, \mu_m$ are the categorical distributions for each component. One may consider the grouped sample setting to be a generalized version of multinomial mixture models.

\section{Related Work} \label{sec:previous}
A significant amount of work regarding the grouped sample setting has focused on the setting where $\Omega$ is finite, which is equivalent to a multinomial mixture model. Some of the earliest work on identifiability was done on binomial mixture models with \cite{teicher63} demonstrating that binomial mixture models are identifiable if the number of trials $n$ and the number of mixture components $m$ satisfy $n \ge 2m-1$. These results were extended to multinomial mixture models in \cite{kim1984} and \cite{elmore2003}. \cite{rabani14} and \cite{vandermeulen19} introduced estimators for the multinomial components when the $n\ge 2m-1$ bound is met. Turning to the continuous setting, the paper \cite{ritchie21} introduces a method for recovering mixture components in the grouped sample setting when the components are densities on some Euclidean space. That method is furthermore guaranteed to asymptotically recover the components whenever the mixture model is identifiable.  In \cite{wei20} the authors consider the grouped sample setting where the mixture components come from some parametric class of densities and provide results for identifiability and rates of convergence. Other works have investigated continuous nonparametric mixture models, without assuming the grouped sample setting, by assuming a clustering structure \citep{dan18,aragam20,vankadara21,aragam21,aragam22,kivva22}.

Identifiability with linearly independent components given $n\ge 3$ was first established in \cite{allman09} by way of Kruskal's (Factorization) Theorem \citep{kruskal77}. A spectral algorithm for the estimation of models with linearly independent components can be found in \cite{anandkumar14}. This algorithm also has a nonparametric adaptation \citep{song14}.\sloppy

The grouped sample setting can be considered as a  special case of a finite \emph{exchangable sequence} \citep{kallenberg05}. An (infinite) sequence of random variables $\xi_1,\xi_2,\ldots$ is called \emph{exchangable} if
\begin{equation*}
	\left(\xi_1,\ldots,\xi_m \right)\overset{d}{=}\left( \xi_{k_1},\ldots, \xi_{k_m}\right)
\end{equation*}
for every distinct subsequence $\xi_{k_1},\ldots, \xi_{k_m}$. For an infinite sequence de Finetti's Theorem tells us that (for Borel spaces) one can always decompose the distribution of the sequence to be independent, conditioned on some other random variable in a way akin to \eqref{eqn:grouped-sample}, though this random variable is not necessarily discrete as in \eqref{eqn:grouped-sample}. This theorem does not extend to finite sequences, but there have been works investigating the grouped sample setting for continuous mixtures. In \citet{vinayak19} the authors present rates for estimating a continuous version of $\sP$ in the context of binomial mixture models. 

A generalization of Kruskal's Theorem for $d$-way arrays can be found can be found in \cite{sidiropoulos00} and is related to the techniques we use in the proof of Theorem \ref{thm:kident}. The proof technique we use in Theorem \ref{thm:kdet} is completely novel, so far as we know. As far as the contributions of this paper are concerned, Theorems \ref{thm:kident}, \ref{thm:nokident}, and \ref{thm:nokdet} are natural extensions of the grouped sample results using techniques from \cite{vandermeulen19} to the independence setting considered in \cite{sidiropoulos00}. The determinedness result in Theorem \ref{thm:kdet} required the development of a new proof strategy.\sloppy

\section{Main Results}\label{sec:main}
In this section we present the main results of this paper. They are related to a property which we call \emph{$k$-independence}.
\begin{defin}
  A sequence of vectors $x_1,\ldots,x_m$ is called \emph{$k$-independent} if every subsequence $x_{i_1},\ldots,x_{i_k}$ containing $k$ vectors is linearly independent.
\end{defin}
The concept of \emph{$k$-independence} is simply a generalization of \emph{Kruskal rank} \citep{kruskal77} to vector spaces. We define $k$-independence using a \emph{sequence} rather than a \emph{set} of vectors so as to relate it to a matrix rank, since a matrix can have repeated columns. When $x_1,\ldots,x_m$ are distinct (as will be the case in our main theorems) we can define $k$-independence simply using sets and subsets.  We now present the main results of this paper.

\begin{thm}\label{thm:kident}
  Let $m\ge2$. If $\sP = \sum_{i=1}^m a_i \delta_{\mu_i}$ is a mixture of measures where $\mu_1,\ldots,\mu_m$ are $k$-independent then $\sP$ is $n$-identifiable if $2m-1\le (k-1)n$.
\end{thm}

\begin{thm}\label{thm:kdet}
  Let $m \ge 2$. If $\sP = \sum_{i=1}^m a_i \delta_{\mu_i}$ is a mixture of measures where $\mu_1,\ldots,\mu_m$ are $k$-independent then $\sP$ is $n$-determined for even-valued $n$ if $2m-2\le (k-1)(n-1)$.
\end{thm}
For Theorem \ref{thm:kident} we must omit the case where $m=1$ since this would imply that $k=1$ which results in the inequality ``$1 \le 0$'' in the theorem statement. Note that any mixture containing only one component is trivially $1$-identifiable, which is accounted for by row one in Table \ref{tab:oldresults}.

The following theorem demonstrates that Theorem \ref{thm:kident} cannot be improved for any values of $m,n$, or $k$, not satisfying $2m-1\le (k-1)n$.
\begin{thm}\label{thm:nokident}
  For all $m\ge k\ge2$ and $n$  with $ 2m-1 > (k-1)n$ there exists a mixture of measures $\sP = \sum_{i=1}^m a_i \delta_{\mu_i}$ where $\mu_1,\ldots,\mu_m$ are $k$-independent and $\sP$ is not $n$-identifiable.
\end{thm}
For determinedness we have a similar bound, but it is a bit loose.
\begin{thm}\label{thm:nokdet}
  For all $m\ge k\ge2$ and $n$ with $ 2m > (k-1)n$ there exists a mixture of measures $\sP = \sum_{i=1}^m a_i \delta_{\mu_i}$ where $\mu_1,\ldots,\mu_m$ are $k$-independent and $\sP$ is not $n$-determined.
\end{thm}
\subsection{Comparison to Previous Results}\label{ssec:comparison}
The results in Section \ref{sec:main} are quite general and contain four of the five bounds from \cite{vandermeulen19} as special cases. Since any pair of distinct probability measures are not colinear, it follows that the components of any mixture of measures with at least two mixture components are $2$-independent. Setting $k=2$ in Theorems \ref{thm:kident} and \ref{thm:kdet} gives us the first two bounds from Table \ref{tab:oldresults} (noting that $n$ is always even for the determinedness result).\sloppy

If a collection of $m$ vectors are linearly independent we have that they are $m$-independent. Setting $k=m$ in Theorem \ref{thm:kident} we have
\begin{align*}
	2m-1 \le (m-1)n 
	\iff \frac{2m-1}{m-1} \le n
	\iff \frac{2m-2+1}{m-1} \le n
	\iff \frac{2m-2}{m-1}+ \frac{1}{m-1} \le n,
\end{align*}
with the minimal $n$ satisfying this being $3$ which yields row 3 in Table \ref{tab:oldresults}. The analogous determinedness bound on row 4 can similarly be derived from Theorem \ref{thm:kdet},

\begin{align*}
  2m-2 \le (m-1)(n-1) \Rightarrow 2 \le n-1 \iff 3 \le n,
\end{align*} 
and the smallest even-valued $n$ satisfying this bound is $4$. 
As a final point we remark that, in contrast to previous results, Theorems \ref{thm:kident} and \ref{thm:kdet} imply that, when $n=3$ or $n=4$ respectively, it is possible to have identifiability/determinedness without linearly independent components; for example by setting $n=4$, $k=7$, and $m=10$ for the determinedness case.
\subsection{Applications} \label{sec:app}
The grouped sample setting occurs naturally in many problem settings including group anomaly detection \citep{muandet13}, transfer learning \citep{blanchard11}, and distribution regression/classification \citep{poczos13,szabo14}. In these settings one has access to groups of samples $\bX_1,\ldots,\bX_N$ with $\bX_i = \left(X_{i,1},\ldots,X_{i,n}\right)$. Mathematical analysis of techniques in this setting typically assume $n\to \infty$. The study of such problems for fixed $n$ is less explored, and the results here can help give some intuition for the learnability of this setting. 

Our results are particularly relevant when samples $X_{i,j}$ lie in a finite sample space, $\left|\Omega\right|= d< \infty$. When $\left|\Omega \right| = \infty$ one could convert samples to a finite sample space by assigning them to a set of $d$ prototypes. In this setting it is natural to assume that the mixture components are $d$-independent due to the following proposition.
\begin{prop}\label{prop:liprob}
  Let $\Psi$ be a measure which is absolutely continuous to the uniform measure on the probability simplex $\Delta^{d-1}$ and let $\Gamma_1,\ldots, \Gamma_d \simiid \Psi$. Then $\Gamma_1,\ldots, \Gamma_d$ are linearly independent with probability one.
\end{prop}
This fact is particularly relevant for topic modeling where $d$, the number of words in a vocabulary, can be large and estimation can be difficult. A straightforward way to fix this is to assign words to $d'<d$ clusters, perhaps using a vector word embedding \citep{mikolov2013}, thereby coarsening the event space. To recover $m$ topics we would should have $d'\le m$ and satisfy $2m-1 \le (d'-1)n$ where $n$ is the number of words per document. To test whether a corpus could potentially contain more topics than a proposed topic model with $m$ topics, we would need that $2m-2\le (d'-1)(n-1)$. These results are also useful for other discrete clustering problems \citep{portela08}.

\section{Proofs}
This section contains proofs of the results in Section \ref{sec:main} and supporting lemmas. Proofs omitted in this section can be found in Appendix \ref{app:proofs}. The $\prod$ symbol represents tensor product when applied to elements of a Hilbert space. For a natural number $N$, $\left[N\right]$ is defined to be $\left\{1,2,\ldots,N\right\}$. To streamline the presentation of our main theorems we first introduce the mathematical tools we will be using.

The following lemma is not particularly novel, but we will be using it quite extensively without reference so we include a statement of it here.
\begin{lem} \label{lem:iplinind}
  Let $x_1,\ldots, x_m$ nonzero be vectors in an inner product space. Then $x_1,\ldots,x_m$ are linearly independent iff there exist vectors $z_1,\ldots ,z_m$ such that $\left<x_i, z_i\right>\neq 0$ for all $i$ and $\left<x_i, z_j\right> = 0$ for all $i\neq j$.
\end{lem}
The next lemma serves as something of a workhorse in our proofs.
\begin{lem} \label{lem:kindpow}
  Let $x_1,\ldots, x_m$ be vectors in a Hilbert space which are $k$-independent with $k\ge2$. Then $x_1^{\otimes n},\ldots,x_m^{\otimes n}$ are $\min \left(n\left(k-1\right)+1,m\right)$-independent.
\end{lem}
\begin{proof}
  We will first consider the case where $n(k-1)+1 = m$. We can relabel the vectors $x_1,\ldots,x_{n(k-1)+1}$ as $x$ and $x_{i,j}$ where $(i,j)\in [n]\times [k-1]$. By $k$-independence, for all $i$, there exists a vector $z_i$ such that $\left<z_i,x\right> = 1$ and $\left<z_i,x_{i,j}\right> = 0$ for all $j$. From this we have that
  \begin{align*}
    \left<x^{\otimes n}, \prod_{i=1}^n z_i\right>=\prod_{i=1}^n \left<x,  z_i\right>= 1 
	  \text{ and }
    \left<x_{i,j}^{\otimes n}, \prod_{l=1}^n z_l\right> = \prod_{l=1}^n\left<x_{i,j},  z_l\right>= 0& \text{, for all $i,j$}.
  \end{align*}
  Because the relabeling is arbitrary it follows that for all $i' \in [n(k-1)+1]$ there exists $\bz_{i'}$ such that $\left<x_{i'}^{\otimes n},\bz_{i'}\right> = 1$  and $\bz_{i'}\perp x_{j'}^{\otimes n}$ for all $j'\neq i'$. Thus we have that $x_1^{\otimes n},\ldots,x_m^{\otimes n}$ are $m$-independent. We will now consider two other cases for the value of $m$. For $m<n(k-1)+1$ we can show that $x_1^{\otimes n},\ldots,x_m^{\otimes n}$ are linearly independent by the same argument. If $m>n(k-1)+1$ then it follows from the $m=n(k-1)+1$ case that every subsequence of length $n(k-1)+1$ of $x_1^{\otimes n},\ldots,x_m^{\otimes n}$ is independent, so it follows that $x_1^{\otimes n},\ldots,x_m^{\otimes n}$ is $\left(n(k-1)+1\right)$-independent.
\end{proof}
The following lemma's proof is very similar to the proof of Lemma \ref{lem:kindpow}, but we defer it to Appendix \ref{app:proofs} due to its length. Note that it recovers Lemma \ref{lem:kindpow} by setting $k' = k$.
\begin{lem}\label{lem:kpindpow}
  Let $x_1,\ldots,x_m$ be $k$-independent with $k\ge2$ and $x$ such that $x,x_1,\ldots,x_m$ is $k'$-independent with $k\ge k'>1$. Then $x^{\otimes n}, x_1^{\otimes n},\ldots,x_m^{\otimes n}$ is $\min(m+1,(n-1)(k-1) + k')$-independent.
\end{lem}
To prove Theorem \ref{thm:kident} we will use the following slight adaptation of Kruskal's Theorem.
\begin{thm}[Hilbert space extension of \cite{kruskal77}]\label{thm:kruskalhilb}
  Let $x_1,\ldots,x_r$, $y_1,\ldots,y_r$, and $z_1,\ldots,z_r$ be elements of three Hilbert spaces $\sH_x,\sH_y,\sH_z$ such that $x_1,\ldots,x_r$ are $r_x$-independent with $r_y,r_z$ defined similarly. Further suppose that $r_x+r_y+r_z \ge 2r+2$. If $a_1,\ldots,a_l\in \sH_x, b_1,\ldots, b_l\in \sH_y$, and $c_1,\ldots,c_l\in \sH_z$ with $r\ge l$ such that\sloppy
  \begin{equation*}
    \sum_{i=1}^r x_i\otimes y_i \otimes z_i = \sum_{j=1}^l a_j \otimes b_j \otimes c_j,
  \end{equation*}
  then $l=r$ and there exists a permutation $\sigma: [r] \to [r]$ and $D_x,D_y,D_z \in \rn^r$ such that $a_{\sigma(i)} = x_i D_{x,i}, b_{\sigma(i)} = y_i D_{y,i}$, and $c_{\sigma(i)} = z_iD_{z,i}$ with $D_{x,i}D_{y,i}D_{z,i} = 1$ for all $i$.
\end{thm}
The following three lemmas allow us to embed general measures in Hilbert spaces and will allow us to use tools from Hilbert space theory \citep{kadison83}.
\begin{lem}[\cite{vandermeulen19} Lemma 6.2] \label{lem:himbed}
  Let $\gamma_1,\ldots,\gamma_n$ be finite measures on a measurable space $\left( \Psi, \sG \right)$. There exists a finite measure $\pi$ and nonnegative functions $f_1,\ldots,f_n \in  L^1\left( \Psi,\sG,\pi \right)\cap L^2\left( \Psi,\sG,\pi \right)$ such that, for all $i$ and all $B \in \sG$
  \begin{equation*}
    \gamma_i(B)=\int_B f_i d\pi.
  \end{equation*}
\end{lem}
The last lemma will be used in particular to embed collections of probability measures in a joint measure space as pdfs.
\begin{lem}[\cite{vandermeulen19} Lemma 6.3] \label{lem:radprod}
  Let $\left( \Psi, \sG \right)$ be a measurable space, $\gamma$ and $\pi$ a pair of finite measures on that space, and $f$ a nonnegative function in $L^1\left(\Psi,\sG, \pi \right)$ such that, for all $A \in \sG$, $\gamma\left( A \right)=\int_A f d\pi$. Then for all $n$, for all $B \in \sG^{\times n}$ we have
  \begin{equation*}
    \gamma^{\times n}\left( B \right) = \int_B f^{\times n} d\pi^{\times n}.
  \end{equation*}
\end{lem}

\begin{lem}[\cite{vandermeulen19} Lemma 5.2]
  \label{lem:l2prod}
  Let $\left(\Psi,\sG,\gamma\right)$ be a measure space. There exists a unitary transform $U:L^2\left( \Psi, \sG, \gamma \right)^{\otimes n} \to L^2\left( \Psi^{\times n}, \sG^{\times n}, \gamma^{\times n} \right)$ such that, for all $f_1,\ldots, f_n \in L^2\left( \Psi, \sG, \gamma \right)$,
  \begin{equation*}
    U\left( f_1\otimes \cdots \otimes f_n \right) = f_1(\cdot)\cdots f_n(\cdot).
  \end{equation*}
\end{lem}
Here and elsewhere note that powers of a $\sigma$-algebra utilizes the standard $\sigma$-algebra product.
Finally we remind the reader of the following standard result from real analysis .
\begin{lem}[\cite{folland99} Proposition 2.23] \label{lem:inteq}
  Let $\left( \Psi,\sG,\gamma \right)$ be a measure space and $f,g \in L^1\left( \Psi,\sG,\gamma \right)$. Then $f=g$ $\gamma$-almost everywhere iff, for all $A\in \sG$, $\int_A f d\gamma = \int_A g d\gamma$.\sloppy
\end{lem}
For the rest of the paper we will leave the ``almost everywhere'' qualifier implicit. We can now prove the main theorems in Section \ref{sec:main}.
\begin{proof}\textbf{of Theorem \ref{thm:kident}}
  Let $\sQ=\sum_{i=1}^l b_i \delta_{\nu_i}$ be a mixture of measures with $l \le m$, such that $V_n(\sP) = V_n(\sQ)$. From this we have that
  \begin{equation*}
    \sum_{i=1}^m a_i \mu_i^{\times n} = \sum_{j=1}^l b_j \nu_j^{\times n}.
  \end{equation*}
  From Lemma \ref{lem:himbed} there exists a measure $\xi$ and nonnegative functions $p_1,\ldots,p_m,q_1,\ldots,q_l \in L^1(\xi)\cap L^2(\xi)$, such that, for all measurable $A$ and $i$, $\mu_i(A) = \int_A p_i d\xi$ and $\nu_i(A) = \int_A q_i d\xi$.
  From Lemmas \ref{lem:radprod} and \ref{lem:inteq} we have that
  \begin{equation*}
    \sum_{i=1}^m a_i p_i^{\times n} = \sum_{j=1}^l b_j q_j^{\times n},
  \end{equation*}
  and from Lemma \ref{lem:l2prod} we have
  \begin{equation}\label{eqn:identhilb}
    \sum_{i=1}^m a_i p_i^{\otimes n} = \sum_{j=1}^l b_j q_j^{\otimes n}.
  \end{equation}
  From the theorem hypothesis we know that $m\ge 2$, and trivially $k \le m$, so we have the following
  \begin{align*}
    2m-1\le (k-1)n \Rightarrow 2m-1\le (m-1)n \Rightarrow\frac{2m-1}{m-1} \le n \Rightarrow 2 < n \Rightarrow 3 \le n.
  \end{align*}
  Because $n\ge 3$ it is always possible to decompose $n=n_1+n_2+n_3$ where $n_i$ are all positive integers.

  We will now prove the following claim which we will denote ``\dag'': if the sequences $p_1^{\otimes n_i},\ldots, p_m^{\otimes n_i}$ (for all $i\in[3]$) are $k_i$-independent respectively and $k_1+k_2+k_3\ge 2m+2$ then the theorem conclusion follows. From (\ref{eqn:identhilb}) we have that\sloppy
  \begin{equation*}
    \sum_{i=1}^m a_i p_i^{\otimes n_1} \otimes p_i^{\otimes n_2} \otimes p_i^{\otimes n_3}  = \sum_{j=1}^l b_j q_j^{\otimes n_1} \otimes q_j^{\otimes n_2} \otimes q_j^{\otimes n_3}.
  \end{equation*}

  From Theorem \ref{thm:kruskalhilb} we have that $l = m$, there exists $D_1,D_2,D_3 \in \rn^m$, and a permutation $\sigma :[m] \to [m]$, such that, for all $i$
  \begin{align}
    b_{\sigma(i)}q_{\sigma(i)}^{\otimes n_1} &=  a_i p_i^{\otimes n_1}D_{1,i}\notag\\
    q_{\sigma(i)}^{\otimes n_2} &=  p_i^{\otimes n_2}D_{2,i}\label{eqn:d2}\\
    q_{\sigma(i)}^{\otimes n_3} &=  p_i^{\otimes n_3}D_{3,i} \notag
  \end{align}
  where $D_{1,i}D_{2,i}D_{3,i} =1$ for all $i$. Applying Lemma \ref{lem:l2prod} to (\ref{eqn:d2}) we have that, for all $i$
  \begin{align*}
    \int q_{\sigma(i)}^{\times n_2} d\xi^{\times n_2} = \int D_{2,i} p_i^{\times n_2}d\xi^{\times n_2} \Rightarrow 1 = D_{2,i}.
  \end{align*}
  So $D_2$ is a vector of ones and so is $D_3$ by the same argument. We have that $D_1$ is also a vector of ones since $D_{1,i}D_{2,i}D_{3,i} =1$ for all $i$. Thus we have that $p_i = q_{\sigma(i)}$ for all $i$. Assuming that $\sigma$ is the identity mapping it follows that $a_i = b_i$ and we have shown \dag.

  Now that \dag  has been demonstrated, to finish the proof we will show that we can decompose $n=n_1+n_2+n_3$ such that $p_1^{\otimes n_i},\ldots, p_m^{\otimes n_i}$ are $k_i$-independent for each $i$ with $k_1+k_2+k_3 \ge 2m+2$ which will finish our proof. To continue we will split into the cases where $n \mod 3$ is 0, 1, or 2.\\
  \textbf{Case 0:}
  We have that $n= 3n'$ for some positive integer $n'$ and we can let $n_1=n_2=n_3 = n'$. Now we can reformulate our Hilbert space embedding:
  \begin{align*}
    \sum_{i=1}^m a_i p_i^{\otimes n} = \sum_{i=1}^m a_i p_i^{\otimes n'}\otimes p_i^{\otimes n'}\otimes p_i^{\otimes n'}.
  \end{align*}
  From Lemma \ref{lem:kindpow} we know that $p_1^{\otimes n'},\ldots,p_m^{\otimes n'}$ are $\min\left((k-1)n'+1,m\right)$-independent. If $\min\left((k-1)n'+1,m\right)$ is $m$ then it follows that $k_1+k_2+k_3 = 3m \ge 2m+2$ since $m\ge2$. If $\min\left((k-1)n'+1,m\right)= (k-1)n'+1$ then we have that $k_1+k_2+k_3 = 3(k-1)n'+3 = (k-1)n + 3 \ge 2m+2$ by the theorem hypothesis $(k-1)n\ge 2m-1$.\sloppy\\
  \textbf{Case 1:} Here we have that $n = 3n' + 1$ and we let $n_1 = n_2 = n'$ and $n_3= n'+1$, so
  \begin{align*}
    \sum_{i=1}^m a_i p_i^{\otimes n} = \sum_{i=1}^m a_i p_i^{\otimes n'}\otimes p_i^{\otimes n'}\otimes p_i^{\otimes n'+1}.
  \end{align*}
  From Lemma \ref{lem:kindpow} we have that $k_1=k_2 = \min(m,(k-1)n'+1)$ and $k_3 = \min(m,(k-1)(n'+1)+1)$.
  If we have that $\min(m,(k-1)n'+1) = m$ then it follows that $\min(m,(k-1)(n'+1)+1) = m$ and $k_1+k_2+k_3 = 3m\ge 2m+2$. 

  If we have that $\min(m,(k-1)n'+1) =(k-1)n'+1$ and $\min(m,(k-1)(n'+1)+1) = (k-1)(n'+1)+1$ then we have that
  \begin{align*}
    k_1+k_2+k_3
    &=2((k-1)n'+1) + (k-1)(n'+1)+1
    =(3n' + 1 )(k-1) +3\\
    &=n(k-1) +3
    \ge 2m+2,
  \end{align*}
  by the theorem hypothesis.

  Lastly, if we have that $\min(m,(k-1)n'+1) = (k-1)n'+1$ and $\min(m,(k-1)(n'+1)+1)
  = m$ with $m < (k-1)(n'+1)+1$ (for equality see the above $\min(m,(k-1)(n'+1)+1)
  =(k-1)(n'+1)+1$ case) then we have that
  \begin{align*}
    &m < (k-1)(n'+1)+1\\
    \Rightarrow& m < (k-1)2n'+1\\
    \Rightarrow &m \le (k-1)2n'\\
    \Rightarrow &2m + 2 \le 2((k-1)n' + 1) + m
    = k_1 + k_2 + k_3,
  \end{align*}
  where the second inequality follows because $n' \ge 1$ and the third inequality follows because both sides of the inequality are integers.
  \\\textbf{Case 2:} We have that $n = 3n'+2$, so let $n_1 = n'$ and $n_2=n_3 = n'+1$. Now we have that
  \begin{align*}
    \sum_{i=1}^m a_i p_i^{\otimes n} = \sum_{i=1}^m a_i p_i^{\otimes n'}\otimes p_i^{\otimes n'+1}\otimes p_i^{\otimes n'+1}.
  \end{align*}
  If $\min(m,(k-1)n' +1) = \min(m,(k-1)(n'+1) +1) = m$ and thus $m=k_1=k_2=k_3$, we have that $k_1+k_2+k_3 \ge 2m+2$ as in the previous cases. If we have that $k_2=k_3=\min(m,(k-1)(n'+1) +1) = (k-1)(n'+1)+1$ then it also follows that $k_1 = \min(m,(k-1)n' +1)=(k-1)n'+1$ and
  \begin{equation*}
    k_1+k_2+k_3 = (k-1)(3n'+2) +3 = n(k-1)+3 \ge 2m+2.
  \end{equation*}
  Now if we have that $k_2=k_3 = m$ and $k_1 =(k-1)n' + 1$ then $k_1+k_2+k_3 = 2m + (k-1)n'+1$ and since $k\ge 2$ and $n' \ge 1$ we have that $k_1+k_2+k_3 \ge 2m +2$ so we are done, and with this final case have finished the proof.
\end{proof}

\begin{proof}\textbf{Sketch of Theorems \ref{thm:nokident} and \ref{thm:nokdet}}
  Theorem 4.2 in \cite{vandermeulen19} states that, for all $m \ge 2$, there exists a mixture of measures $\sP = \sum_{i=1}^m a_i \delta_{\mu_i}$ which is not $2m-2$-identifiable. Therefore there exists a mixture of measures $\sQ = \sum_{i=1}^{l} b_i \delta_{\nu_i}$ with $\sP\neq\sQ$, and $l \le m$ such that
  \begin{align*}
    \sum_{i=1}^m a_i \mu_i^{\times 2m-2} = \sum_{j=1}^{l} b_j \nu_j^{\times 2m-2}.
  \end{align*}
  Since $(k-1)n \le 2m-2$ we have that either $(k-1)n = 2m-2$ or, if $(k-1)n < 2m-2$ we can apply Lemma \ref{lem:noident}, giving us
  \begin{align*}
    \sum_{i=1}^m a_i \mu_i^{\times (k-1)n} = \sum_{j=1}^m b_j \nu_j^{\times (k-1)n}.
  \end{align*}
  Since any pair of distinct probability measures are linearly independent, it follows that any collection of probability measures are $2$-independent. Using this fact we can can adapt Lemma \ref{lem:kindpow} to show that $\mu_1^{\times k-1},\ldots ,\mu_m^{\times k-1}$ are $k$-independent. Letting $\sP' = \sum_{i=1}^m a_i \delta_{\mu_i^{\times k-1}}$ and $\sQ' = \sum_{i=1}^{l} b_i \delta_{\nu_i^{\times k-1}}$, we have that $V_n(\sP') = V_n(\sQ')$ and we are done. The proof of Theorem \ref{thm:nokdet} is virtually identical and follows from \cite{vandermeulen19} Theorem 4.4.
\end{proof}

\begin{proof}\textbf{of Theorem \ref{thm:kdet}}
   If $n=2$  then
  \begin{align*}
    &2m-2 \le \left(k-1\right) \left(n-1\right) \\
    \Rightarrow &2m-1 \le k \le m \quad \text{(noting $k\le m$)} \\
    \Rightarrow& m\le 1 \Rightarrow m=1.
  \end{align*}
   Since the theorem hypothesis assumes $m\ge2$ we have that the theorem is vacuously true for the $n=2$ case. \footnote{Its worth noting that one does indeed have determinedness for $n=2$ and $m=1$ from Table \ref{tab:oldresults} row two.} To finish the proof we will show that theorem holds for $n=4$ and then proceed by induction.\\
  \textbf{Base Step}: Let $n=4$. We will proceed by contradiction and assume there exists $k\ge2$ and a collection of $k$-independent probability measures $\mu_1,\ldots,\mu_m$ such that there exists a mixture of measures $\sP = \sum_{i=1}^m a_i \delta_{\mu_i}$ and $\sQ= \sum_{i=1}^l b_i \delta_{\nu_i}$ a mixture of measures where $\sP \neq \sQ$ and $V_4(\sP) = V_4(\sQ)$.
  From the theorem hypothesis that $m \ge 2$ it follows that $k\ge2$ and $k-2 \ge0$. Applying this bound to our theorem hypothesis with $n=4$ we have that
  \begin{align*}
    2m-2 \le (k-1)3
    \Rightarrow 2m-1 \le 3k-2 \le 3k-2 + (k-2) = 4k-4 = (k-1)4 = (k-1)n.
  \end{align*}
  Since we have that $2m-1 \le (k-1)n$ we can apply Theorem \ref{thm:kident} and thus $l > m$.  We will proceed analogously to the proof of Theorem \ref{thm:kident} and embed the measures in a Hilbert space as before
  \begin{equation}
    \sum_{i=1}^m a_i p_i^{\otimes 4} = \sum_{j=1}^l b_j q_j^{\otimes 4}.\label{eqn:smiley}
  \end{equation}
  Because $l> m$ there exists $i$ such that $q_i \neq p_j$ for all $j$.
  We will assume without loss of generality that $q_1$ satisfies this. Let $k'$ be the largest value such that $q_1,p_1,\ldots,p_m$ are $k'$-independent.
  We will now show that that $k'<k$. To see this suppose that $k'\ge k$ which would imply that $q_1,p_1,\ldots,p_m$ are $k$-independent. Observe that $m>2$ (thus $m\ge 3$, we use this at \eqref{eqn:m3}). Were this not the case then the components of $\sP$, $\mu_1$ and $\mu_2$, would be linearly independent and $\sP$ would be 4-determined from Table \ref{tab:oldresults} row 4, thereby violating the contradiction hypothesis. With the base case $n=4$ in our theorem hypothesis, and the fact that $k$ and $m$ are positive integers, we get
  \begin{align}
    2m-2 \le 3(k-1)
    &\Rightarrow 2m \le 3k-1\notag\\
    &\Rightarrow \frac{4}{3}m \le 2k-\frac{2}{3}\notag\\
    &\Rightarrow m \le 2k-\frac{2}{3}- \frac{1}{3}m\notag\\
    &\Rightarrow m \le 2k-\frac{2}{3}- 1\quad\text{(using $m\ge 3 $)}\label{eqn:m3}\\
    &\Rightarrow m \le 2k- 2 \quad \text{($m$ is an integer)}\notag\\
    &\Rightarrow m+1 \le 2(k-1)+1. \label{eqn:kpindpow}
  \end{align}
  From application of Lemma \ref{lem:kindpow} it follows that $q_1^{\otimes 2},p_1^{\otimes 2 },\ldots,p_m^{\otimes 2 }$ are $\min(2(k-1)+1,m+1)$-independent and from (\ref{eqn:kpindpow}) it follows that they are linearly independent. Now we have that there exists $z$ such that $z\perp p_i^{\otimes 2}$ for all $i$ but $\left<z,q_1^{\otimes 2}\right> = 1$ and thus
  \begin{align}
    0= \left<z^{\otimes 2},\sum_{i=1}^m a_i p_i^{\otimes 4}\right> \label{eqn:detli}
    = \left<z^{\otimes 2},\sum_{j=1}^l b_j q_j^{\otimes 4}\right>
    =\sum_{j=1}^l b_j \left<z, q_j^{\otimes 2}\right>^2 > 0,
  \end{align}
  a contradiction. So $k'<k$.

  Because $k'<k$ there must exist a collection of $k'$ elements of $p_1,\ldots,p_m$, which we denote $p_{i_1},\ldots,p_{i_{k'}}$, such that $q_1,p_{i_1},\ldots,p_{i_{k'}}$ are linearly dependent; for convenience we will assume without loss of generality that these elements are $p_1,\ldots,p_{k'}$. Because $p_1,\ldots,p_{k'}$ are linearly independent but $q_1,p_1,\ldots,p_{k'}$ are linearly dependent we have that $q_1 = \sum_{i=1}^{k'}\alpha_i p_i$ for some $\alpha_1,\ldots,\alpha_{k'}$.

  We will now show that $k' \ge 2k-m+1$.
  Suppose this were not the case and $k'\le 2k-m$ or equivalently $m \le 2k-k'$. By $k$-independence there exists a vector $z$ such that $\left<z,p_{k'}\right> = 1$ with $z \perp p_1,\ldots,p_{k'-1}, p_{k'+1},\dots, p_{k}$ and another vector $z'$ such that $\left<z',p_1\right> = 1$ and $z' \perp p_2,\ldots,p_{k'},p_{k+1},\ldots,p_{m}$. We know $z'$ exists since $k'\le 2k-m$ so the cardinality of $p_2,\ldots,p_{k'},p_{k+1},\ldots,p_{m}$ satisfies the following\sloppy
  \begin{align*}
    \left|\left\{2,\ldots,k'\right\}\right|+\left|\left\{k+1,\ldots,m \right\} \right|
    =k'-1 + m-k \le 2k-m -1 +m -k =k-1.
  \end{align*}
  Now we have that
  \begin{align*}
    &\left<z^{\otimes 2}\otimes z'^{\otimes 2}, \sum_{i=1}^m a_i p_i^{\otimes 4} \right>
    =\sum_{i=1}^m a_i \left<z,p_i\right>^2\left<z',p_i\right>^2\\
    &=\sum_{i\in\left\{1,\ldots,k'-1,k'+1,\ldots,k\right\}} a_i \left<z,p_i\right>^2\left<z',p_i\right>^2
    + \sum_{i\in\left\{k',k+1,\ldots,m\right\}} a_i \left<z,p_i\right>^2\left<z',p_i\right>^2\\
    &=\sum_{i\in\left\{1,\ldots,k'-1,k'+1,\ldots,k\right\}} a_i 0\left<z',p_i\right>^2
    + \sum_{i\in\left\{k',k+1,\ldots,m\right\}} a_i \left<z,p_i\right>^2 0\\
    &=0.
  \end{align*}
  On the other hand we have that
  \begin{align*}
    \left<z^{\otimes 2}\otimes z'^{\otimes 2}, \sum_{i=1}^l b_i q_i^{\otimes 4} \right>
    &= \sum_{i=1}^l b_i \left<z,q_i\right>^2\left<z',q_i\right>^2
    \ge b_1 \left<z,q_1\right>^2\left<z',q_1\right>^2\\
    &= b_1 \left<z,\sum_{i=1}^{k'}\alpha_i p_i\right>^2\left<z',\sum_{i=1}^{k'}\alpha_i p_i\right>^2\\
    &= b_1 \left(\sum_{i=1}^{k'}\alpha_i\left<z, p_i\right>\right)^2\left(\sum_{i=1}^{k'}\alpha_i\left<z', p_i\right>\right)^2\\
    &= b_1 \alpha_{k'}^2\alpha_1^2
    >0
  \end{align*}
  which contradicts (\ref{eqn:smiley}). Thus we have that $k'\ge 2k-m+1$.

  We are now going to show that $q_1^{\otimes 2},p_1^{\otimes 2 },\ldots,p_m^{\otimes 2 }$ are linearly independent via Lemma \ref{lem:kpindpow}. To do this we will show that $(2-1)(k-1)+ k' \ge m+1$. Using $k' \ge 2k-m+1$ we have that 

  \begin{align*}
    2m-2 &\le (4-1)(k-1) \quad \text{(base case hypothesis with $n=3$)}\\
    \Rightarrow  2m &\le 3k-1\\
    \Rightarrow  m + 1 &\le 3k -m\\
    \Rightarrow  m + 1 &\le (k-1) +(2k-m+1)\\
    \Rightarrow  m + 1 &\le (k-1) +k'.
  \end{align*}
  Since $q_1^{\otimes 2},p_1^{\otimes 2},\ldots,p_m^{\otimes 2}$ are linearly independent we can finish our contradiction using the same argument as in (\ref{eqn:detli}).

  \textbf{Induction Step}:
  We will now proceed by induction along $n$ in an increment of $2$ since the theorem statement holds for even-valued $n$. For our inductive hypothesis assume that for even valued $n\ge 4$ and all $k,m$ with $2m-2 \le (n-1)(k-1)$ that any mixture of measures with $m$ components which are $k$-independent are $n$-determined. Consider some mixture of measures $\sP = \sum_{i=1}^{m'}a_i\delta_{\mu_i}$ with $k$-independent components and $2m'-2 \le (k-1)((n+2)-1)$. If $k=m'$ then we have that the components are linearly independent so it is $(n+2)$-determined by Lemma \ref{lem:ident} and Table \ref{tab:oldresults} row four. Now suppose that $m'>k$. Let $\sQ$ be a mixture of measures with $l$ components such that $V_{n+2}(\sP) = V_{n+2}(\sQ)$. Embedding $V_{n+2}\left(\sP\right)$ and $V_{n+2}\left(\sQ\right)$ as before we have that
  \begin{align*}
    \sum_{i=1}^{m'} a_i p_i^{\otimes n+2} = \sum_{j=1}^l b_j q_j^{\otimes n+2}.
  \end{align*}
  By $k$-independence there exists $z$ such that $\left<z,p_1\right> \neq 0$ and $z \perp p_{m'-(k-1)+1},\ldots, p_{m'}$. Now we have that
  \begin{align}
    \sum_{i=1}^{m'} a_i p_i^{\otimes n+2} = \sum_{j=1}^l b_j q_j^{\otimes n+2}
    &\Rightarrow \sum_{i=1}^{m'} a_i p_i^{\otimes n}\left<p_i^{\otimes 2},\cdot \right> = \sum_{j=1}^l b_j q_j^{\otimes n}\left<q_j^{\otimes 2},\cdot\right> \label{eqn:kad}\\
    &\Rightarrow \sum_{i=1}^{m'} a_i p_i^{\otimes n}\left<p_i^{\otimes 2},z^{\otimes 2} \right> = \sum_{j=1}^l b_j q_j^{\otimes n}\left<q_j^{\otimes 2},z^{\otimes 2}\right>\notag\\
    &\Rightarrow \sum_{i=1}^{m'-(k-1)} a_i p_i^{\otimes n}\left<p_i,z \right>^2 = \sum_{j=1}^l b_j q_j^{\otimes n} \left<q_j,z \right>^2\notag.
  \end{align}
  The implication (\ref{eqn:kad}) follows from the equivalence between tensor products and Hilbert-Schmidt operators (see \cite{kadison83} Proposition 2.6.9). Let $\lambda = \sum_{i=1}^{m'-(k-1)}a_i\left<p_i,z\right>^2$, $a_i' = a_i\left<p_i,z\right>^2/\lambda$, and $b_i'= b_i\left<q_i,z\right>^2/\lambda$. Without loss of generality we will assume that $\left<q_i,z\right> \neq 0$ for $i\in[l']$ and $\left<q_i,z\right> = 0$ for $i>l'$, with $l'$ potentially equaling $l$. Note that $a_i'\ge0$ and $\sum_{i=1}^{m'-(k-1)}a_i'=1$ and likewise for $b_i'$, since the right hand side of \eqref{eqn:bprime-eq} is a convex combination of pdfs that itself must be equal to a pdf. So now we have

  \begin{equation}\label{eqn:bprime-eq}
    \sum_{i=1}^{m'-(k-1)} a_i' p_i^{\otimes n} = \sum_{j=1}^{l'} b_j' q_j^{\otimes n}. 
  \end{equation}
  Note that
  \begin{align*}
    2m'-2 \le (k-1)((n+2)-1)
    \Rightarrow& 2(m'-1) \le 2(k-1) + (k-1)(n-1)\\
    \Rightarrow& 2(m' - (k-1))-2 \le (k-1)(n-1)
  \end{align*}
  so by the induction hypothesis we have that the mixture of measures $\sum_{i=1}^{m'-(k-1)} a_i' \delta_{\mu_i}$ is $n$-determined\footnote{There is a somewhat suble point here that two measures are equal if they admit the same measure. So even though some components in the mixture of measures $\sum_{i=1}^{m'-(k-1)} a_i' \delta_{\mu_i}$ may have a zero coefficient, it is still a mixture of measures in the sense that was described in Section \ref{ssec:problem-setting}, although it may have fewer than $m'-(k-1)$ mixture components.}. It follows that
  $\sum_{i=1}^{m'-(k-1)} a_i' \delta_{\mu_i} = \sum_{j=1}^{l'} b_j' \delta_{\nu_j}$. Without loss of generality we will assume that $\mu_1 = \nu_1$ and $a_1' = b_1'$. It follows that $p_1=q_1$ and thus $a_1=b_1$. By the same argument it follows that $\nu_i=\mu_i$ and $a_i = b_i$ for all $i$ and, because $\sum_{i=1}^{m'}b_i = 1$, that $m'=l$ and thus $\sP = \sQ$ and $\sP$ is $n+2$-determined, which finishes our proof.
\end{proof}

\acks
{Robert A. Vandermeulen acknowledges support by the German Federal Ministry of Education and Research (BMBF) for the Berlin Institute for the Foundations of Learning and Data (BIFOLD) (01IS18037A). Ren\'e Saitenmacher acknowledges support by the German Federal Ministry of Education and Research (BMBF) in the project Patho234 (763031LO207).
}

\bibliographystyle{plain}
\bibliography{rvdm}
\appendix
\section{Proofs}\label{app:proofs}
\begin{proof}\textbf{of Lemma \ref{lem:iplinind}}
  For the forward direction, since $x_1,\ldots,x_m$ are linearly independent we can find the associated $z_1,\ldots,z_m$ from the Gram-Schmidt process. We prove the other direction by contradiction: suppose $x_1,\ldots,x_m$ are not linearly independent but there exist $z_1,\ldots,z_m$ satisfying the property in the lemma statement. From this it follows (without loss of generality) that $x_1 =\sum_{i=2}^m\alpha_i x_i$. We also know that there exists $z_1$ such that $\left<x_1,z_1\right> = 1$ but $\left<x_i,z_1\right> =0$ for all $i\ge 2$. Then we have that
  \begin{align*}
    1= \left<x_1,z_1\right>
    = \left<\sum_{i=2}^m\alpha_i x_i,z_1\right>
    = \sum_{i=2}^m\alpha_i\left< x_i,z_1\right>
    =0,
  \end{align*}
  a contradiction.
\end{proof}

\begin{proof}\textbf{of Lemma \ref{lem:kpindpow}}
  We first prove the case where $m+1 = (n-1)(k-1) + k'$ or equivalently $m=(n-1)(k-1)+(k'-1)$. We will use Lemma \ref{lem:iplinind} to demonstrate the linear independence of $x,x_1,\ldots,x_m$. First we will show that there exists a tensor which is perpendicular to $x^{\otimes n}$ and all but one of the vectors in $x_1^{\otimes n},\ldots,x_m^{\otimes n}$. To do this we relabel $x_1,\ldots,x_m$ to $x_{i,j}$ for $\left(i,j\right) \in [n-1]\times[k-1]$ and $x'_1,\ldots,x'_{k'-1}$. From $k$-independence we can find $z_1,\ldots,z_{n-1}$ such that $\left<z_i,x_{i,j}\right> = 0$ for all $i,j$ and $\left<z_i,x'_1\right> = 1$. Likewise, from $k'$-independence there exists $z$ such that $\left<z,x'_1\right> = 1$, $\left<z,x'_i\right> = 0$ for all $2\le i \le k'-1$, and $\left<z,x\right>=0$. Now we have that
  \begin{align*}
    \left<{x'}_1^{\otimes n}, z\otimes \prod_{i=1}^{n-1} z_i \right> =\left<x'_1,z\right> \prod_{i=1}^{n-1}\left<x'_1, z_i \right> &= 1\\
    \left<{x'}_i^{\otimes n}, z\otimes \prod_{j=1}^{n-1} z_j \right> = \left<x'_i,z\right> \prod_{j=1}^{n-1}\left<x'_i, z_j \right> &= 0 \quad \forall i\ge2\\
    \left<x_{i,j}^{\otimes n}, z\otimes \prod_{l=1}^{n-1} z_l \right> = \left<x_{i,j},z\right> \prod_{i=1}^{n-1}\left<{x}_{i,j}, z_l \right> &= 0 \quad \forall (i,j)\\
    \left<x^{\otimes n}, z\otimes \prod_{i=1}^{n-1} z_i \right>= \left<x,z\right> \prod_{i=1}^{n-1}\left<x, z_i \right> &= 0.
  \end{align*}
  Because $x'_1$ was arbitrary due to relabeling, there exist tensors $\bz_1,\ldots,\bz_m$ such that $\left<x_i^{\otimes n},\bz_i\right> = 1$ for all $i$, $\left<x_j^{\otimes n}, \bz_i\right> = 0$ for all $j\neq i$, and $\left<x^{\otimes n},\bz_i\right> = 0$ for all $i$.

  We will now find a tensor which is perpendicular to all $x_1^{\otimes n},\ldots,x_m^{\otimes n}$ but is not perpendicular to $x^{\otimes n}$, which will complete our proof of the $m+1 = (n-1)(k-1) + k'$ case. If $x^{\otimes n-1}\notin \spn\left(\left\{x_1^{\otimes n-1},\ldots,x_{(n-1)(k-1)+1}^{\otimes n-1} \right\} \right)$ then there exists a vector $\bz$ such that $\left<x^{\otimes n-1},\bz\right>=1$ and $\left<x_i^{\otimes n-1}, \bz\right> = 0$ for all $i \in [(n-1)(k-1)+1]$. By $k'$-independence there exists a vector $z$ such that $\left<x,z\right>=1$ and $\left<x_j,z\right>=0$ for all $j\in\left\{(n-1)(k-1)+1,\ldots, (n-1)(k-1)+k'-1\right\}$. From this it follows that $\left<x^{\otimes n},\bz\otimes z \right> = 1$ but $\left<x_i^{\otimes n}, \bz\otimes z\right> =0$ for all $i$.

  Now we will assume that $x^{\otimes n-1}\in \spn\left(\left\{x_1^{\otimes n-1},\ldots,x_{(n-1)(k-1)+1}^{\otimes n-1} \right\} \right)$. Note that $x_1^{\otimes n-1},\ldots,x_{(n-1)(k-1)+1}^{\otimes n-1}$ are linearly independent by Lemma \ref{lem:kindpow}. From this it follows that there exists exactly one linear combination of $x_1^{\otimes n-1},\ldots,x_{(n-1)(k-1)+1}^{\otimes n-1}$ which is equal to $x^{\otimes n-1}$. We assume without loss of generality that $x_{(n-1)(k-1)+1}^{\otimes n-1}$ has a nonzero coefficient in that solution. Now have that $x^{\otimes n-1},x_1^{\otimes n-1},\ldots , x_{(n-1)(k-1)}^{\otimes n-1}$ are linearly independent. From this there exists $\bz$ such that $\left<\bz, x_i^{\otimes n-1}\right> = 0$ for all $i\in[(n-1)(k-1)]$ but $\left<\bz,x^{\otimes n-1}\right> = 1$. By $k'$-independence there exists $z$ such that $z\perp x_{(n-1)(k-1)+1},\ldots,x_{(n-1)(k-1)+(k'-1)} $ but $\left<z,x\right> = 1$. We now have that $z\otimes \bz\perp x_i^{\otimes n}$ for all $i$ and $\left<z\otimes \bz , x^{\otimes n }\right> = 1$ which finishes the $m+1 = (n-1)(k-1) + k'$ case.\sloppy

  We will now take care of the other cases for the values of $m$. If $m+1 < (n-1)(k-1) + k'$ then $x^{\otimes n},x_1^{\otimes n},\ldots, x_m^{\otimes n}$ are linearly independent by the same argument made in the $m+1 = (n-1)(k-1) + k'$ case so $x^{\otimes n},x_1^{\otimes n},\ldots, x_m^{\otimes n}$ are $m+1$-independent. 

  If $m+1 > (n-1)(k-1) + k'$ we would like to show that any subsequence of $x^{\otimes n},x_{1}^{\otimes n},\ldots,x_{m}^{\otimes n}$ containing $(n-1)(k-1)+k'$ vectors is linearly independent. We consider two cases, where a subsequence contains $x^{\otimes n}$ or it does not. If it does then $x^{\otimes n},x_{i_1}^{\otimes n},\ldots,x_{i_{(n-1)(k-1)+k'-1}}^{\otimes n}$ is linearly independent from the $m+1= (n-1)(k-1) + k'$ case and replacing $x^{\otimes n}$ with $x_{i_{(n-1)(k-1)+k'}}^{\otimes n}$ leaves this sequence linearly independent since $k \ge k'$ so $x^{\otimes n}, x_1^{\otimes n},\ldots, x^{\otimes n}_m$ is $((n-1)(k-1)+k')$-independent, thus finishing the proof.  \sloppy
\end{proof}

\begin{proof}\textbf{of Theorem \ref{thm:kruskalhilb}}
  Note that two Hilbert spaces with the same finite dimension are isometric to one another. Suppose that $a_1,\ldots,a_l$, $b_1,\ldots, b_l$ and $c_1,\ldots,c_l$ with $m\le l$ such that
  \begin{equation*}
    \sum_{i=1}^r x_i\otimes y_i \otimes z_i = \sum_{j=1}^l a_j \otimes b_j \otimes c_j.
  \end{equation*}
  Let $\overline{\sH}_x = \spn\left(\left\{x_1,\ldots, x_r, a_1,\ldots a_l \right\} \right)$ with $\overline{\sH}_y$ and $\overline{\sH}_z$ defined similarly. Because these spaces are finite dimensional the theorem follows from direct application of Kruskal's Theorem, see the following. 
\end{proof}
\begin{definition}
  A matrix $M$ has \emph{Kruskal rank} $k$ if every collection of $k$ columns of $M$ are linearly independent.
\end{definition}
The following theorem is a statement Kruskal's Theorem \citep{kruskal77} adapted from \citep{rhodes09}.
\begin{thm}[Kruskal's Theorem]
  For a matrix $M$ let $M_i$ be its $i$th column vector. Let $A,B,C$ be matrices of dimensions $d_A\times r,d_B\times r$, and $d_C\times r$ respectively with Kruskal rank $k_A,k_B,k_C$ respectively and let $k_A+k_B+k_C \ge 2r+2$. Let $F,G,H$ be matrices with dimensions $d_A \times s, d_B \times s,d_C \times s$ with $s\le r$ and 
  \begin{equation*}
    \sum_{i=1}^r A_i\otimes B_i \otimes C_i = \sum_{j=1}^s F_j\otimes G_j\otimes H_j.
  \end{equation*}
  Then there exists a permutation matrix $P$ and invertible diagonal matrices $D_A,D_B,D_C$ such that $D_A D_B D_C = I_r$ such that
  \begin{align*}
    F = AD_AP\\
    G = BD_BP\\
    H = CD_CP.
  \end{align*}
\end{thm}

\begin{proof}\textbf{of Proposition \ref{prop:liprob}}
  Let $\Gamma_1,\ldots,\Gamma_d \simiid \Psi$. We will proceed by contradiction and assume that $\Gamma_1,\ldots,\Gamma_d$ are linearly dependent with nonzero probability. It follows that that $\Gamma_1 = \sum_{i=2}^{d} \alpha_i \Gamma_i$ for some $\alpha_2,\ldots,\alpha_d$ with nonzero probability. Let $1_d$ be the $d$-dimensional vector containing all ones. Because $\Gamma_1,\ldots,\Gamma_d$ are all probability vectors it follows that
  \begin{align*}
    1_d^T\Gamma_1 &= 1^T_d\sum_{i=2}^m \alpha_i \Gamma_i \\
    \Rightarrow 1 &= \sum_{i=2}^d \alpha_i.
  \end{align*}
  The probabilistic simplex lies in an affine subspace of dimension $d-1$ which we will call $S$. Because of this there exists an affine operator $f$ which is a bijection between $S$ and a closed subset of $\rn^{d-1}$ with $f(x) = Mx+b$ for some matrix $M$ and vector $b$. Let $\tGamma_i = f(\Gamma_i)$. We have that $\tGamma_i \sim\Psi\left(f^{-1}(\cdot)\right)$ is a measure on $\rn^{d-1}$ which is absolutely continuous wrt the Lebesgue measure on $\rn^{d-1}$ and thus $\tGamma_1,\ldots,\tGamma_d$ lie in general position with probability one (see \cite{devroye97} Section 4.5).
  Note that
  \begin{align*}
    \tGamma_1
    &= M\Gamma_1+b
    = M\left(\sum_{i=2}^d \alpha_i \Gamma_i \right)+\left(\sum_{j=2}^d \alpha_j \right)b
    = \sum_{i=2}^d \alpha_i M  \Gamma_i+  \alpha_i b\\
    &= \sum_{i=2}^d \alpha_i \left(M  \Gamma_i+   b\right)
    = \sum_{i=2}^d \alpha_i \tGamma_i.
  \end{align*}
  Since $\tGamma_2,\ldots,\tGamma_d$ trivially lie in a $\left(d-2\right)$-dimensional affine subspace there exists a vector $v\neq 0_{d-1}$ and $r$ such that $v^T\tGamma_i = r$ for $i\ge2$.
  Now we have that 
  \begin{align*}
    v^T\sum_{i=2}^d \alpha_i \tGamma_i = \sum_{j=1}^d \alpha_i r \Rightarrow v^T\tGamma_1 = r
  \end{align*}
  and thus, with nonzero probability, $\tGamma_1,\ldots,\tGamma_d$ do not lie in general position, a contradiction.
\end{proof}

\begin{proof}\textbf{of Theorem \ref{thm:nokident}}
  From \cite{vandermeulen19} Theorem 4.2, for all $m \ge 2$  there exists a mixture of measures $\sP = \sum_{i=1}^m a_i \delta_{\mu_i}$ which is not $2m-2$-identifiable, thus there exists a mixture of measures $\sQ = \sum_{i=1}^{m'} b_i \delta_{\nu_i}\neq \sP$ with $m' \le m$ and
  \begin{align*}
    \sum_{i=1}^m a_i \mu_i^{\times 2m-2} = \sum_{j=1}^{m'} b_j \nu_j^{\times 2m-2}.
  \end{align*}
  Since $(k-1)n \le 2m-2$ we have that either $(k-1)n = 2m-2$ and it directly follows that
  \begin{align*}
     \sum_{i=1}^m a_i \mu_i^{\times (k-1)n} = \sum_{j=1}^{m'} b_j \nu_j^{\times (k-1)n}
  \end{align*}
  or that $(k-1)n < 2m-2$ and we have
  \begin{align*}
    &\sum_{i=1}^m a_i \mu_i^{\times 2m-2} = \sum_{j=1}^{m'} b_j \nu_j^{\times 2m-2}\\
    &\Rightarrow\sum_{i=1}^m a_i \mu_i^{\times (k-1)n}\times \mu_i^{\times 2m-2 - (k-1)n} = \sum_{j=1}^{m'} b_j \nu_j^{\times (k-1)n} \times \nu_j^{\times 2m-2 - (k-1)n}\\
    &\Rightarrow\sum_{i=1}^m a_i \mu_i^{\times (k-1)n}\times \mu_i^{\times 2m-2 - (k-1)n}(\Omega^{\times 2m-2 - (k-1)n})\quad \text{ (essentially Lemma \ref{lem:noident})}\\
    &\qquad= \sum_{j=1}^{m'} b_j \nu_j^{\times (k-1)n} \times \nu_j^{\times 2m-2 - (k-1)n}\left(\Omega^{\times 2m-2 - (k-1)n}\right)\\
    &\Rightarrow \sum_{i=1}^m a_i \mu_i^{\times (k-1)n} = \sum_{j=1}^{m'} b_j \nu_j^{\times (k-1)n}.
  \end{align*}
  If we let $\sP'= \sum_{i=1}^m a_i \delta_{\mu_i^{\times k-1}}$ and $\sQ' = \sum_{i=1}^{m'}b_i \delta_{\nu_i^{\times k-1}}$ then we have that $V_n(\sP')= V_n(\sQ')$. To finish the proof we will show that $\mu_1^{\times k-1},\ldots, \mu_m^{\times k-1}$ are $k$-independent and we are done. 
  To do this we will proceed by contradiction, suppose that they are not $k$-independent and there exists a nontrivial linear combination of $k$ elements in $\mu_1,\ldots,\mu_m$ which is equal to zero. We will assume $\mu_1,\ldots,\mu_k$ without loss of generality satisfy this, so
  $$\sum_{i=1}^k \alpha_i \mu_i^{\times k-1}=0$$
  and there exists $i$ such that $\alpha_i \neq 0$.
  Embedding these measures as was done in the proof of Theorem \ref{thm:kident} we have that
  \begin{align*}
    \sum_{i=1}^k \alpha_i p_i^{\otimes k-1} = 0
  \end{align*}
  but since any pair of distinct $p_i,p_j$ are $2$-independent, applying Lemma \ref{lem:kindpow} gives us that $p_1^{\otimes k-1}, \ldots, p_k^{\otimes k-1}$ are linearly independent, a contradiction.\sloppy
\end{proof}

\begin{proof}\textbf{ of Theorem \ref{thm:nokdet}}
  From \cite{vandermeulen19} Theorem 4.4, for all $m\ge 1$ there exists a mixture of measures $\sP=\sum_{i=1}^m a_i \delta_{\mu_i}$ which is not $2m-1$-determined. From here this proof proceeds exactly as the proof of Theorem \ref{thm:nokident}.
\end{proof}

\end{document}